\documentclass{amsart}
\usepackage{amssymb}

\usepackage[all]{xy}

\newcommand{\e}{\varepsilon}
\newcommand{\w}{\omega}
\newcommand{\IR}{\mathbb R}
\newcommand{\IN}{\mathbb N}
\newcommand{\C}{\mathcal C}

\newtheorem{theorem}{Theorem}
\newtheorem{lemma}{Lemma}
\newtheorem{claim}{Claim}
\newtheorem{corollary}{Corollary}
\newtheorem{proposition}{Proposition}
\theoremstyle{definition}
\newtheorem{definition}{Definition}
\newtheorem{remark}{Remark}

\title{Returning functions with closed graph are continuous}
\author{Taras Banakh, Ma\l gorzata Filipczak, Julia W\'odka}
\address{T.Banakh: Ivan Franko National University of Lviv (Ukraine) and Jan Kochanowski University in Kielce (Poland)}
\email{t.o.banakh@gmail.com}
\address{M. Filipczak: Faculty of Mathematics and Computer Sciences, Lodz University, ul. Stefana Banacha 22, 90-238 \L\'od\'z (Poland)}
\email{malgorzata.Flipczak@wmii.uni.lodz.pl}
\address{J.~W\'odka: L\'od\'z University of Technology, Institute of Mathematics, ul. Wólcza\'nska 215, 90-924 L\'od\'z (Poland)}
\email{juliawodka@gmail.com}
\subjclass{26A15, 54C08, 54D05}
\keywords{Continuous function, returning function, closed graph, path-inductive space}

\begin{document}
\begin{abstract} A function $f:X\to \IR$ defined on a topological space $X$ is called {\em returning} if for any point $x\in X$ there exists a positive real number $M_x$ such that for every path-connected subset $C_x\subset X$ containing the point $x$ and any $y\in C_x\setminus\{x\}$ there exists a point $z\in C_x\setminus\{x,y\}$ such that $|f(z)|\le \max\{M_x,|f(y)|\}$. A topological space $X$ is called {\em path-inductive} if a subset $U\subset X$ is open if and only if for any path $\gamma:[0,1]\to X$ the preimage $\gamma^{-1}(U)$ is open in $[0,1]$. The class of path-inductive spaces includes all first-countable locally path-connected spaces and all sequential locally contractible space. 
 We prove that a function $f:X\to \IR$ defined on a path-inductive space $X$ is continuous if and only of it is returning and has closed graph. This implies that a (weakly) \'Swi\c atkowski function $f:\IR\to\IR$ is continuous if and only if it has closed graph, which answers a problem of Maliszewski, inscibed to Lviv Scottish Book.
\end{abstract} 
\maketitle

Let $X$ and $Y$ be a topological spaces. We say that function $f\colon
X\rightarrow Y$ has {\em closed graph} if its graph $$\Gamma_f:=\{\langle x,f(x)\rangle:x\in X\}$$ is closed
in the product $X\times Y$. It is well-known that each continuous function $f:X\to Y$ to a Hausdorff topological space $Y$ has closed graph. Trivial examples show that the converse statement is not true in general. 

There exist many (algebraic or topological) properties of functions, which being combined with the closedness of the graph yield the continuity, see Tao's blog \cite{Tao}. For example, by the classical Closed Graph Theorem, a linear operator between Banach spaces is continuous if and only if it has closed graph.

Known topological properties implying the continuity of functions with closed graph include the  Darboux property, the subcontinuity, the almost continuity, the near continuity, the $B$-quasicontinuity, the weak Gibson property, etc. (see \cite{Alas}, \cite{BB}, \cite{Ber}, \cite{Bors}, \cite{Burg}, \cite{BP}, \cite{DN}, \cite{HH}, \cite{LG}, \cite{MN}, \cite{Moors}, \cite{PS}, \cite{W}).

The inspiration for our present investigation was a problem, written by the third author to the {\tt Lviv 	Scotish Book}\footnote{http://www.math.lviv.ua/szkocka/} on 01.05.2018. Actually, the problem was originally asked several years ago
by Aleksander Maliszewski: \textit{Is any \'{S}wi\c{a}tkowski function with
	closed graph continuous?} Our main theorem gives an affirmative answer to
this question and a posteriori to many its modifications (for peripherally continuous functions,
peripherally bounded functions,  Darboux functions, etc.)

One of the most general properties, responsible for the continuity of real-valued functions with closed graph is introduced in the following definition.

\begin{definition} A real-valued function $f:X\to \IR$ on a topological space $X$ is defined to be 
\begin{itemize}
\item {\em returning at a point $x\in X$} if there exists a positive real number $M_x$ such that for every path-connected subset $C_x\subset X$ containing the point $x$ and any $y\in C_x\setminus\{x\}$ there exists a point $z\in C_x\setminus\{x,y\}$ such that $|f(z)|\le \max\{M_x,|f(y)|\}$;
\item {\em returning} if $f$ is returning at each point.
\end{itemize}
\end{definition}

%A function $f:\IR\to \IR$ is called
%\begin{itemize}
%\item {\em Swi\c atkowski} if for any real numbers $a<b$ with $f(a)\ne f(b)$ there exists a continuity point $x\in I(a,b)$ of $f$ such that $f(x)\in I(f(a),f(b))$;
%\item {\em weak Swi\c atkowski} if for any real numbers $a<b$ with $f(a)\ne f(b)$ there exists a point $x\in I(a,b)$ such that $f(x)\in I[f(a),f(b)]$.
%\end{itemize}

The following theorem is the main result of this paper.

\begin{theorem}\label{main} A function $f:\IR\to\IR$ is continuous if and only if $f$ has closed graph and is returning.
\end{theorem}

For the proof of this theorem we shall need some auxiliary notions and results.

For two points $a<b$ on the real line let $[a,b]$ and $(a,b)$ be the closed and open intervals with end-points $a,b$, respectively. To avoid confusion between notations for open intervals and ordered pairs, an ordered pair of real numbers $a,b$ will be denoted by $\langle a,b\rangle$.

A function $f:X\to\IR$ defined on a topological space $X$ is called 
\begin{itemize}
\item {\em weakly discontinuous} if for any non-empty closed subset $A\subset\IR$ the set $C(f{\restriction}A)$ of continuity points of the restriction $f{\restriction}A$ has non-empty interior in $A$;
\item {\em locally bounded} at a point $x\in X $ if $x$ has a neighborhood $U_x\subset X$ such that the set $f(U_x)$ is bounded in $\IR$. 
\end{itemize}

The following lemma is known (see \cite{Baggs}, \cite{Burg}) and its proof is included for the convenience of the reader.

\begin{lemma}\label{l1} Assume that a function $f:X\to\IR$, defined on a closed subset $X\subset\IR$ has closed graph. Then  
\begin{enumerate}
\item $f$ is weakly discontinuous;
\item $f$ is continuous at a point $x\in X$ if and only if $f$ is locally bounded at $x$;
\item the set $C(f)$ of continuity points is open and dense in $X$.
\end{enumerate}
\end{lemma}

\begin{proof} For every $n\in\IN$ consider the set $K_n=\{x\in X:\max\{|x|,|f(x)|\}\le n\}$ and observe that it is compact, being the projection of the compact subset $\Gamma_f\cap [-n,n]^2$ onto the real line. Next, observe that the restriction $f{\restriction}K_n$ is continuous since it has compact graph $\Gamma_f\cap [-n,n]^2$. For every closed subset $A\subset X$, the Baire Theorem yields a number $n\in\IN$ such that $A\cap K_n$ contains a non-empty relatively open subset $U$ of $A$. Taking into account that $U=C(f{\restriction}U)\subset C(f{\restriction}A)$, we see that $f$ is weakly discontinuous.
\smallskip

Assuming that $f$ is locally bounded at some point $x\in X$, we can find a bounded neighborhood $U_x\subset X$ such that $f(U_x)$ is bounded and hence $U_x\cup f(U_x)\subset (-n,n)$. Then $U_x\subset K_n$ and the continuity of $f{\restriction}U_x$ follows from the continuity of $f{\restriction}K_n$.
\smallskip

The weak discontinuity of $f$ implies the density of the set $C(f)$ in $X$. To see that $C(f)$ is open in $X$, take any continuity point $x\in X$ of $f$ and find a neighborhood $U_x\subset X$ whose image $f(U_x)$ is bounded in the real line. By Lemma~\ref{l1}(2), $f{\restriction}U_x$ is continuous and hence $U_x\subset C(f)$.
\end{proof}

\begin{proof}[Proof of Theorem~\ref{main}] The ``only if'' part is trivial. To prove the ``if'' part, assume that $f:\IR\to\IR$ is a returning function with closed graph. By Lemma~\ref{l1}(3), the set $C(f)$ of continuity points of $f$ is open and dense in $\IR$. If $C(f)=\IR$, then $f$ is continuous and we are done. So, assume that $C(f)\ne\IR$.

Let $\C$ be the family of connected components of the set $C(f)$. It follows that each set $C\in\C$ is an open connected subset of the real line.

\begin{claim}\label{cl1} For each $C\in\C$ and its closure $\bar C$ in $\IR$ the restricted function $f{\restriction}\bar C$ is continuous. 
\end{claim}

\begin{proof} Since $C(f)\ne\IR$, the set $C$ is of one of three types: $(-\infty,b)$, $(a,+\infty)$ or $(a,b)$ for some real numbers $a,b$. First assume that $C=(-\infty,b)$ for some $b\in\IR$. We need to prove that $f{\restriction}(-\infty,b]$ is continuous. 
By Lemma~\ref{l1}(2), it suffices to show that  $f{\restriction}(-\infty,b]$ is locally bounded at $b$.

To derive a contradiction, assume that $f|(-\infty,b]$ is not locally bounded at $b$. We claim that $$\liminf_{x\to b-0}|f(x)|=\infty.$$ Assuming that $\liminf_{x\to b-0}|f(x)|<\infty$, we can fix any $M>\max\{|f(b)|,\liminf_{x\to b-0}|f(x)|\}$ and conclude that  for every $\e>0$ the interval $(b-\e,b)$ contains a point $x$ such that $|f(x)|<M$. 

Since the graph $\Gamma_f$ is closed and the points $\langle b,M\rangle$ and $\langle b,-M\rangle$ do not belong to $\Gamma_f$, there exists $\e>0$ such that the set $[b-\e,b]\times\{-M,M\}$ is disjoint with $\Gamma_f$. Since $f{\restriction}(-\infty,b]$ is not locally bounded at $b$, there exists a point $y\in (b-\e,b)$ such that $|f(y)|>M>\liminf_{x\to b-0}|f(x)|$. Now the definition of $\liminf_{x\to b-0}|f(x)|<M$ yields a point $x\in (y,b)$ such that $|f(x)|<M<|f(y)|$.  By the Mean Value Theorem, for some point $z\in[x,y]\subset (-\infty,b)\subset C(f)$ we have $|f(z)|=M$ and hence $\langle z,f(z)\rangle\in \Gamma_f\cap \big([b-\e,b]\times\{-M,M\}\big)$, which contradicts the choice of $\e$. This contradiction completes the proof of the equality $\liminf_{x\to b-0}|f(x)|=\infty$.

Since the function $f$ is returning at $b$, there exists a positive real constant $M$ such that for every $\lambda<b$ there exists a point $x\in (\lambda,b)$ such that $|f(x)|\le\max\{M,|f(\lambda)|\}$. Since $\liminf_{x\to b-0}|f(x)|=\infty$, there exists a point $c\in (-\infty,b)$ such that $|f(c)|>M$. Taking into account that the function $f{\restriction}(-\infty,b)$ is continuous and tends to infinity at $b$, we conclude that the set $L:=\{x\in [c,b]:|f(x)|\le|f(c)|\}$ is compact and hence has the largest element $\lambda\in L$.  By the choice of the constant $M$, there exists a point $x\in (\lambda,b)$ such that $|f(x)|\le\max\{M,|f(\lambda)|\}\le |f(c)|$. Then $x\in L$ and hence $x\le\lambda$, which contradicts the inclusion $x\in (\lambda,b)$.
This contradiction completes the proof of the continuity of $f{\restriction}(-\infty,b]$.
 
 By analogy we can consider the cases $C=(a,+\infty)$ and $C=(a,b)$ for some $a<b$.
 \end{proof}
 
 \begin{claim}\label{cl2} Any distinct sets $C_1,C_2\in\C$ have disjoint closures.
 \end{claim}
 
 \begin{proof} Assuming that $\bar C_1\cap\bar C_2\ne \emptyset$, we conclude that the set $\bar C_1\cup\bar C_2$ is convex and so is its interior $U$. The continuity of the restrictions $f{\restriction}\bar C_1$ and $f{\restriction}\bar C_2$ established in Claim~\ref{cl1} implies that $f{\restriction}U$ is continuous and hence $U\subset C(f)$ and $C_1\cup C_2\subset U\subset C$ for some connected component $C\in\C$, which is not possible as distinct connected components of $C(f)$ are disjoint.
 \end{proof}
 
 Claim~\ref{cl2} implies that the complement $F=\IR\setminus C(f)$ has no isolated points. By Lemma~\ref{l1}, the function $f$ is weakly discontinuous, so there exist points $a<b$ such that $\emptyset\ne [a,b]\cap F\subset C(f{\restriction}F)$. By the continuity of $f$ on the compact set $K:=[a,b]\cap F$, there exists a number $M\in\IN$ such that $|f(x)|<M$ for all $x\in K$. The restriction $f{\restriction}[a,b]$ is discontinuous and hence unbounded. Then we can choose a sequence $(x_n)_{n\in\w}$ of points of $[a,b]$ such that $|f(x_n)|>M+n$ for all $n\in\w$. For every $n\in\w$ let $C_n\in\C$ be the unique connected component of $C(f)$, containing the point $x_n$.    Replacing $x_n$ by a suitable subsequence, we can assume that $|f(x_n)|>\max\{|f(x)|:x\in [a,b]\cap\bigcup_{k<n}\bar C_k\}$ for all $n\in\w$. This ensures that the components $C_n$ are pairwise disjoint and hence $\mathrm{diam}(C_n)\to 0$.
 
 Replacing $(x_n)_{n\in\w}$ by a suitable subsequence, we can assume that $C_n\subset [a,b]$ for all $n\in\w$ and  $(x_n)_{n\in\w}$ converges to some point $c\in [a,b]$, which is a discontinuity point of $f$ as $\lim_{n\to\infty}|f(x_n)|=\infty\ne |f(c)|$. Then $|f(c)|<M$ by the choice of $M$. The graph $\Gamma_f$ of $f$ is closed and hence is disjoint with the set $[c-\e,c+\e]\times\{-M,M\}$ for some $\e>0$. For every $n\in\w$ write the component $C_n$ as $C_n=(a_n,b_n)$ for some points $a_n<b_n$ in the set $[a,b]\cap F$. Since $x_n\to c$ and $|b_n-a_n|\to 0$, there exists $n\in\w$ such that $(a_n,b_n)\subset [c-\e,c+\e]$.
 
  Since $|f(a_n)|<M<|f(x_n)|$, the Mean Value Theorem applied to the continuous function $f{\restriction}[a_n,x_n]$ yields a point $z_n\in[a_n,z_n]$ with $|f(z_n)|=M$. Then $\langle z_n,f(z_n)\rangle\in\Gamma_f\cap \big([c-\e,c+\e]\times\{-M,M\}\big)=\emptyset$, which is a desired contradiction completing the proof of Theorem~\ref{main}.
 \end{proof}  
 
 Theorem~\ref{main} admits a generalization to real-valued functions defined on  path-inductive topological spaces. 
\smallskip

By a {\em path} in a topological space $X$ we understand any continuous function $\gamma:[0,1]\to X$.

We define a topological space $X$ to be {\em path-inductive} if a subset $U\subset X$ is open if and only if for any path $\gamma:[0,1]\to X$ the preimage $f^{-1}(U)$ is open in $[0,1]$.

\begin{proposition}\label{p:pathi} A topological space $X$ is path-inductive if $X$ is either sequential and locally contractible or $X$ is first-countable and locally path-connected.
\end{proposition} 
 
 \begin{proof} Assume that $X$ is either sequential and locally contractible or $X$ is first-countable and locally path-connected.
  Given a non-open set $A\subset X$  we should find a path $\gamma:[0,1]\to X$ such thar $f^{-1}(A)$ is not open in $[0,1]$.
 
By the sequentiality or the first-countability of $X$, there exists a sequence $\{x_n\}_{n\in\w}\subset X\setminus A$ that converges to a point $x_\w\in A$. 
 
If $X$ is locally contractible, then there exists a neighborhood $V\subset X$ of the point $x_\w$ and a continuous map $h:V\times[0,1]\to X$ such that $h(x,0)=x$ and $h(x,1)=x_\w$ for all $x\in V$. Replacing $(x_n)_{n\in\w}$ by a suitable subsequence, we can assume that $\{x_n\}_{n\in\w}\subset V$. Consider the homotopy $H:V\times [0,2]\to X$ defined by 
$$H(x,t)=\begin{cases}
h(x,1-t)&\mbox{if $t\in[0,1]$}\\
h(x,t-1)&\mbox{if $t\in[1,2]$}
\end{cases}\mbox{ \ for $(x,t)\in V\times[0,2]$},
$$
and observe that $H(x,0)=H(x,2)=x_\w$ and $H(x,1)=x$ for every $x\in V$.

Define a path $\gamma:[0,1]\to X$ by $\gamma(0)=x_\w$ and $\gamma(t)=H(x_n,2^{n+2}t-2)$ where $n\in\w$ is the unique number such that $\frac1{2^{n+1}}<t\le \frac1{2^n}$. It is clear that the function $\gamma$ is continuous at each point $t\in(0,2]$. To see that $\gamma$ is continuous at $0$, fix any neighborhood $W\subset X$ of $\gamma(0)=x_\w$ and using the continuity of the function $H$ at points of the compact set $\{x_\w\}\times[0,2]\subset H^{-1}(x_\w)\subset H^{-1}(W)$, find a neighborhood $U\subset V$ of the point $x_\w$ such that $H(U\times[0,2])\subset W$. Since the sequence $(x_n)_{n\in\w}$ converges to $x_\w$, there exists a number $m\in\w$ such that $\{x_n\}_{n\ge m}\subset U$. Then $$\gamma([0,\tfrac1{2^m}])\subset \{x_\w\}\cup\bigcup_{n\ge m}H(\{x_n\}\times[0,2])\subset H(U\times[0,2])\subset W,$$witnessing that $\gamma$ is continuous at $0$. 
\smallskip

Next, assume that the space $X$ is first-countable and locally path-connected.
In this case we can find a countable neighborhood base $\{U_n\}_{n\in\w}$ at $x_\w$ such that for every $n\in\w$ and point $x\in U_{n+1}$ there exists a path $\gamma_x:[0,1]\to U_n$ such that $\gamma_x(0)=x$ and $\gamma_x(1)=x_\w$.
Since the sequence $(x_k)_{k\in\w}$ converges to $x_\w$, for every $n\in\w$ there exists a number $k_n\in\w$ such that $x_{k_n}\in U_{n+1}$. By our assumption, there exists a path $\gamma_n:[\tfrac1{2^{n+1}},\frac1{2^n}]\to U_n$ such that $\gamma_n(\frac1{2^{n+1}})=\gamma_n(\frac1{2^{n}})=x_\w$ and $\gamma_n(\frac3{2^{n+2}})=x_{k_n}$. The paths $\gamma_n$, $n\in\w$, compose a continuous path $\gamma:[0,1]\to X$ such that $\gamma(0)=x_\w$ and $\gamma{\restriction}[\frac1{2^{n+1}},\frac1{2^n}]=\gamma_n{\restriction}[\frac1{2^{n+1}},\frac1{2^n}]$ for all $n\in\w$.
\smallskip

Taking into account that $0\in \gamma^{-1}(x_\w)\subset \gamma^{-1}(A)$ and $\frac3{2^{n+2}}\in\gamma^{-1}(x_{k_n})\subset \gamma^{-1}(X\setminus A)=X\setminus\gamma^{-1}(A)$ for all $n\in\w$, we see that the set $\gamma^{-1}(A)$ is not open in $[0,1]$.
\end{proof}

\begin{remark} Proposition~\ref{p:pathi} implies that each sequential linear topological space over the field $\IR$ is path-inductive (being locally contractible). In particular, the inductive limit $\mathbb R^\infty=\varinjlim\mathbb R^n$ of an increasing sequence of finite-dimensional Euclidean spaces is a sequential path-inductive space, which is not first-countable. On the other hand, the Sierpi\'nski triangle is first-countable locally path-connected but not locally contractible.
\end{remark}

Theorem~\ref{main} implies the following its self-generalization.

\begin{theorem}\label{t:main2} A real-valued function $f:X\to\IR$ on a path-inductive topological space $X$ is continuous if and only if $f$ is a returning function with closed graph.
 \end{theorem}
 
 \begin{proof} The ``only if'' part is trivial. To prove the ``if'' part, assume that $f$ has closed graph and is returning. If $f$ is not continuous, then by the path-inductivity of $X$, there exists a path $\gamma:[0,1]\to X$ such that the function $f\circ \gamma$ is not continuous. Let $g:\IR\to\IR$ be the extension of $f\circ\gamma$ such that $g\big((-\infty,0]\big)=\{g(0)\}$ and $g\big([1,+\infty)\big)=\{g(1)\}$. The closedness of the graph of the function $f$ implies the closedness of the graph of the functions $f\circ\gamma$ and $g$. The returning property of the function $f$ implies the returning property of $g$. 
By Theorem~\ref{main}, the function $g$ is continuous, which contradicts the choice of $\gamma$. This contradiction completes the proof.
 \end{proof}
 
Theorems~\ref{main} and \ref{t:main2} motivate studying returning functions in more details. We shall show that the class of returning functions on the real line is quite wide and contains many known classes of real-valued functions possessing some generalized continuity properties.

Let us recall that a function $f:X\to \IR$ on a topological space $X$ is called
\begin{itemize}
\item {\em \'Swi\c atkowski} if for any connected subset $C\subset X$ and points $a,b\in C$ with $f(a)< f(b)$  there exists a continuity point $x\in C$ of $f$ such that $f(a)<f(x)<f(b)$;
\item {\em weakly \'Swi\c atkowski} if for  
any connected subset $C\subset X$ and points $a,b\in C$ with $f(a)< f(b)$  there exists a point $x\in C$ such that $f(a)<f(x)<f(b)$;
\item {\em Darboux} if for any connected subset $C\subset X$ the image $f(C)$ is connected;
\item {\em almost continuous} (or else {\em nearly continuous}) if for any open set $V\subset X$ the preimage $f^{-1}(V)$ is contained in the interior of $\overline{f^{-1}(V)}$;
\item {\em quasicontinuous} if for each point $x\in X$, neighborhood $V\subset Y$ of $f(x)$ and neighborhood $O_x\subset X$ of $x$, there exists a non-empty open set $G\subset O_x$ with $f(G)\subset V$;
\item {\em $B$-quasicontinuous} if for each point $x\in X$, neighborhood $V\subset Y$ of $f(x)$ and an open connected set $O\subset X$ with $x\in\overline{O}$, there exists a non-empty open set $G\subset O$ with $f(G)\subset V$;
\item ({\em weakly}) {\em Gibson} if $f(\overline{O})\subset\overline{f(O)}$ for any open (connected) subset $O\subset X$;
%\item {\em $c$-continuous} if for compact set $K\subset \IR$ the preimage $f^{-1}(K)$ is closed in $X$; 
\item {\em peripherally continuous} if for any point $x\in X$ and neighborhoods $O_x\subset X$ and $O_{f(x)}\subset \IR$ of $x$ and $f(x)$ there exists a neighborhood $V_x\subset O_x$ of $x$ such that $f(\partial V_x)\subset O_{f(x)}$.
\item {\em peripherally bounded} if for any point $x\in X$ there exists a bounded set $B\subset \IR$ such that for any neighborhood $O_x\subset X$ of $x$ there exists a neighborhood $V_x\subset O_x$ of $x$ such that $f(\partial V_x)\subset B$.
\end{itemize}
Here by $\partial V_x$ we denote the boundary of $V_x$ in the topological space $X$. For more information on these classes of functions, see the survey of Gibson and Natkaniec \cite{GN}. 

For any function $f:X\to\IR$ these notions relate as follows (by a simple arrow we denote the implications holding under the additional assumption $X=\IR$):
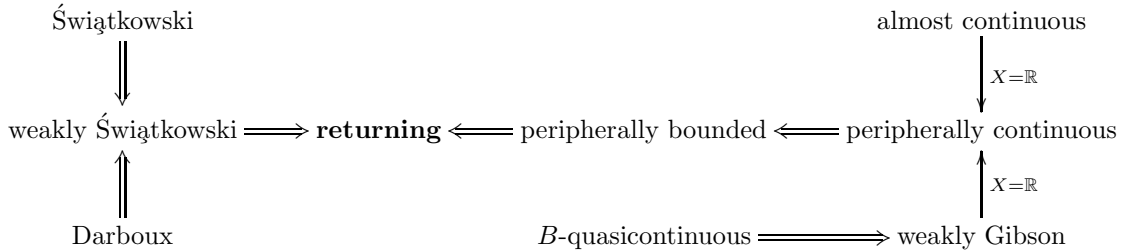
$$
\xymatrix{
\mbox{\'Swi\c atkowski}\ar@{=>}[d]&&&\mbox{almost continuous}\ar^{X=\IR}[d]\\
\mbox{weakly \'Swi\c atkowski}\ar@{=>}[r]&\mbox{\bf returning}
&\mbox{peripherally bounded}\ar@{=>}[l]&\mbox{peripherally continuous}\ar@{=>}[l]\\
\mbox{Darboux}\ar@{=>}[u]&&\mbox{$B$-quasicontinuous}\ar@{=>}[r]&\mbox{weakly Gibson}\ar_{X=\IR}[u]\\
}
$$

Since each (weakly) \'Swi\c atkowski function is returning, Theorem~\ref{main} implies the following corollary that answers the original problem of Maliszewski.

\begin{corollary} A (weakly) \'Swi\c atkowski function $f:\IR\to\IR$ is continuous if and only if it has closed graph.
\end{corollary}

It is interesting to compare Theorems~\ref{main} and \ref{t:main2} to the following known results on the continuity of functions with closed graphs.

\begin{theorem} A function $f:X\to\IR$ defined on a topological space $X$ is continuous if it has closed graph and one of the following conditions is satisfied:
\begin{enumerate}
\item $f$ is almost continuous {\rm (Long and McGehee \cite{LG});}
\item $f$ is nearly continuous and $X$ is Baire {\rm (Moors \cite{Moors});}
\item $f$ is Darboux and $X$ is locally connected {\rm (W\'ojcik \cite[Corollary 18]{WPhD});}
\item $f$ is peripherally continuous and $X=\IR$ {\rm(Hagan \cite{Hagan})};
\item $X=\IR$ and $f$ is bilaterally quasicontinuous {\rm (Dobo\v s \cite{Dobos});}
\item $f$ is $B$-quasicontinuous and $X$ is locally connected {\rm (Borsik \cite{Bors})};
\item $f$ is weakly Gibson {\rm (Das and Nesterenko \cite{DN})}.
\end{enumerate}
\end{theorem}
 % \newpage

%\begin{corollary} For a function $f:\IR\to\IR$ with closed graph the following conditions are equivalent:
%\begin{enumerate}
%\item $f$ is continuous;
%\item $f$ is returning;
%\item $f$ has the \c Swi\c atkowski property;
%\item $f$ has the Darboux property.
%\end{enumerate}
%\end{corollary} 


\begin{thebibliography}{9}
\bibitem{Alas} O.T.~Alas, {\em A note on functions with a closed graph}, Portugal. Math. {\bf 42}:4 (1983/84), 351--354.

\bibitem{Baggs} I.~Baggs, {\em Functions with a closed graph}, Proc. Amer. Math. Soc. {\bf 43} (1974), 439--442.

\bibitem{BB} N.~Bandyopadhyay, P.~Bhattacharyya, {\em Functions with preclosed graphs}, Bull. Malays. Math. Sci. Soc. (2) {\bf 28}:1 (2005), 87--93.

\bibitem{Ber} A.~Berner, {\em Almost continuous functions with closed graphs},  Canad. Math. Bull. {\bf 25}:4 (1982), 428--434. 

\bibitem{Bors} J.~Borsik, {\em Bilateral quasicontinuity in topological spaces}, Tatra Mt. Math. Publ. {\bf 28} (2004), 159--168.

\bibitem{Burg} C.E.~Burgess, {\em Continuous functions and connected graphs}, Amer. Math. Monthly {\bf 97}:4 (1990), 337--339.

	\bibitem{BP} T. Byczkowski, R. Pol, \emph{On the closed graph and open
		mapping theorems}, Bull. L'Acad. Polon. Sci. Ser. Sci. Math., Astronom. et
	Phys. {\bf 24} (1976), 723--726.
	
	%\bibitem{Fer} C.~Fernandez, {\em The closed graph theorem for multilinear mappings}, Internat. J. Math. Math. Sci. {\bf 19}:2 (1996), 407--408.
	
\bibitem{DN} A.~Das, V.~Nesterenko, {\em On decomposition of continuity, $B$-quasicontinuity and closed graph}, preprint.
	
	
\bibitem{Dobos} J.~Dobo\v s, {\em On the set of points of discontinuity for functions with closed graphs}, \v Casopis P\v est. Mat. {\bf 110}:1 (1985),  60--68.
	
	\bibitem{F} R. V. Fuller , \emph{Relations among continuous and various
		non-continuous functions}, Pacific J. Math. {\bf 25} (1968), 459--509.
	
\bibitem{GN} R.~Gibson, T.~Natkaniec, {\em Darboux like functions}, Real Anal. Exchange {\bf 22}:2 (1996/97), 492--533.
	
\bibitem{Hagan} M.~Hagan, {\em Equivalence of connectivity maps and peripherally continuous transformations}, Proc. Amer. Math. Soc. {\bf 17} (1966), 175--177. 
	
	\bibitem{HH} T. R. Hamlett, L.L. Herrington, \emph{The closed graph and
		p-closed graph properties in general topology}, Amer. Mat. Soc. Providence
	Rhode Island, 1981.
	
\bibitem{Jel} J.~Jel\'\i nek, {\em A discontinuous function with a connected closed graph}, Acta Univ. Carolin. Math. Phys. {\bf 44}:2 (2003),  73--77. 

\bibitem{LL} S.Y.~Lin, Y.F.~Lin, {\em On almost continuous mappings and Baire spaces}, Canad. Math. Bull. {\bf 21}:2 (1978),  183--186.

\bibitem{LG} P.~Long, E.~McGehee, {\em Properties of almost continuous functions}, Proc. Amer. Math. Soc. {\bf 24} (1970), 175--180.
	
\bibitem{MN} V.K.~Maslyuchenko, V.~Nesterenko, {\em Decomposition of continuity and transition maps}, Math. Bull. Shevchenko Sci. Soc. {\bf 8} (2011) 132--150.

\bibitem{Moors} W.~Moors, {\em Closed graph theorems and Baire spaces}, New Zealand J. Math. {\bf 31}:1 (2002), 55--62.

	\bibitem{PS} Z. Piotrowski, A. Szyma\'{n}ski, \emph{Closed graph theorem:
		topological approach}, Rend. Circ. Mat. Palermo {\bf 37}:1 (1988), 88--99.
	
\bibitem{Rose} D.A.~Rose, {\em On Levine's decomposition of continuity}, Canad. Math. Bull. {\bf 21}:4 (1978), 477--481.
	
	\bibitem{Tao} T.~Tao, {\em The closed graph theorem in various categories}, {\tt https://terrytao.wordpress.com/2012/11/20}.
	
	\bibitem{W} M. Wilhelm, \emph{On question of B.J. Pettis}, Bull. L'Acad. Polon.
	Sci. Ser. Sci. Math. {\bf 27} (1979), 591--592.
	
\bibitem{WPhD} 	M.R.~W\'ojcik, {\em  Closed and connected graphs of functions;
examples of connected punctiform spaces}, Rozprawa Doktorska, Katowice, 2008; ({\tt https://www.apronus.com/static/MRWojcikPhD.pdf}).
	
\bibitem{WW} M.R.~W\'ojcik, M.S.~W\'ojcik, {\em Separately continuous functions with closed graphs}, Real Anal. Exchange {\bf 30}:1 (2004/05), 23--28. 
	
	\bibitem{WS} M. W\'{o}jtowicz, W. Sieg, \emph{P-spaces and an unconditional
		closed graph theorem}, Rev. R. Acad. Cien. Serie A. Mat. 
		 {\bf 104}:1 (2010), 13--18.
\end{thebibliography}
 \end{document}